\documentclass{amsart}
\usepackage{amsmath,amssymb,amsrefs,amsthm}
\usepackage{tikz, mathrsfs, verbatim, tensor}
\usepackage{enumerate}
\usetikzlibrary{patterns}

\title{Permuton limits for some permutations avoiding a single pattern}

\newtheorem{theorem}{Theorem}[section]
\newtheorem{lemma}[theorem]{Lemma}

\newtheorem{proposition}[theorem]{Proposition}

\newcommand{\utri}{W^+}
\newcommand{\ltri}{W^-}

\newcommand{\souw}{\text{SW}}

\usepackage{tikz}

\usetikzlibrary{decorations.markings}
\tikzstyle{vertex}=[circle, draw, inner sep=0pt, minimum size=6pt]
\tikzstyle{Vertex}=[circle, draw, inner sep=0pt, minimum size=14pt]
\tikzstyle{Vertexc}=[circle, draw, inner sep=0pt, minimum size=14pt, fill=blue!30]
\tikzstyle{vertexc}=[circle, draw, inner sep=0pt, minimum size=6pt, fill=red!40]
\tikzstyle{vertexcg}=[circle, draw, inner sep=0pt, minimum size=6pt, fill=green!70!black]

\newcommand{\beq}{\begin{equation*}}
\newcommand{\eeq}{\end{equation*}}
\newcommand{\ba}{\begin{align*}}
\newcommand{\ea}{\end{align*}}

\newcommand{\field}[1]{\mathbb{#1}}

\newcommand{\matbegin}[1]{\left (  \begin{array} {#1} }
\newcommand{\matend}{ \end{array} \right ) } 
\newcommand{\prob}{\field{P}}

\newcommand{\avn}[1]{\textbf{Av}_n(#1)}
\newcommand{\avne}[2]{\textbf{Av}_{n, #2}(#1)}

\newcommand{\muI}{\mu_I}
\newcommand{\muJ}{\mu_J}

\author{
Kaitlyn~Hohmeier and Erik~Slivken
}

\begin{document}
    \maketitle

\section*{Abstract}
{Permutons are probability measures on the unit square with uniform marginals that provide a natural way to describe limits of permutations. We are interested in the permuton limits for permutations sampled uniformly from certain pattern-avoiding classes that are in bijection with the class of permutations avoiding an increasing pattern of fixed length. In particular, we will look at a family of permutations whose permuton limit collapses to the unique permuton supported on the line $x + y = 1$ in the unit square, informally known as the anti-diagonal.  We prove some general properties about permutons to aid our efforts, which may be useful for proving permuton limits that converge to the anti-diagonal for a broader range of permutation classes.  
}
    
\section{Introduction}

What does a permutation sampled uniformly at random from a class of pattern-avoiding permutations look like? Some early work on this question comes from both Miner and Pak, and separately Madras and Pehlivan \cite{minerpak, madraspehlivan1} whose key result proved the probability that a point of a uniformly random 321-avoiding permutation lands far off the diagonal decays exponentially in $n$.  Hoffman, Rizzolo, and Slivken \cite{hrs1} showed that the points of these permutations converge, with some appropriate scaling, to Brownian excursion and its reflection.  Madras and Pehlivan \cite{madraspehlivan2} and later Hoffman, Rizzolo, and Slivken \cite{hrs4} extended their respective results to permutations avoiding fixed monotone patterns of size larger than $3$.  In particular, Madras and Pehlivan first showed that with high probability, there are no points far away from  the anti-diagonal for a uniformly random permutation avoiding an increasing pattern of fixed size.  

Our main goal is to expand on the results of \cite{madraspehlivan2} using the language of permutons \cite{permuton1,permuton2} in conjunction with a bijection between certain classes of pattern-avoiding permutations introduced and proven by Backelin, West, and Xin \cite{BWX}.  A permuton is a probability measure on the unit square with uniform marginals.  The space of permutons, denoted $\mathcal{M}$, forms a compact metric space in which we can describe weak limits of sequences of permutons.  Specifically, we are interested in the permuton limits for permutations sampled uniformly from certain pattern-avoiding classes that are in bijection with the class of permutations avoiding the increasing pattern of some fixed length.

\subsection{Definitions}

Before stating our main result, we first introduce some definitions and notation. Denote the set of all permutations of size $n$ by $S_n$.  Let $\pi$ be a permutation in $S_m$ and $\sigma$ a permutation in $S_n$, where $m \leq n$.  We say $\sigma $ contains a copy of $\pi$ if there exists a sequence of indices $1 \leq i_1<\cdots<i_m\leq n$ such that the relative order of $\sigma(i_1), \cdots, \sigma(i_m)$ is the same as $\pi$.  If $\sigma$ contains no copies of $\pi$, we say $\sigma$ is $\pi$-avoiding.  We denote the set of $\pi$-avoiding permutations of size $n$ by $\avn{\pi}$.  Both \cite{kitaev} and \cite{vatter} provide thorough introductions to the subject of permutation patterns.  

We let $\sigma^r$, $\sigma^c$, and $\sigma^{rc}$ denote the reverse, complement, and reverse-complement of a permutation $\sigma$, where $\sigma^r(i) = \sigma(n+1-i)$, $\sigma^c(i) = n+1-\sigma(i),$ and $\sigma^{rc}(i) = n+1-\sigma(n+1-i).$  Each of these operations can be viewed as bijections on the set of permutations.  More importantly, they play quite nicely with the subset of pattern avoiding permutations in that $\sigma\in \avn{\pi}$ if and only if $\sigma^{r}\in \avn{\pi^r}$ if and only if $\sigma^{c} \in \avn{\pi^c}$ thus inducing a bijection between any two of the sets $\avn{\pi}, \avn{\pi^r}, \avn{\pi^c},$ and $\avn{\pi^{rc}}$.      

Let $\pi$ and $\pi'$ be two permutations in $S_k$ and $S_{k'}$ respectively.  Their direct sum, denoted $\pi\oplus\pi'$ is a permutation on $S_{k+k'}$ defined point-wise as
\[
\pi\oplus\pi'(i) = \left \{ 
\begin{array}{lr}
\pi(i), & 1\leq k,\\
k+\pi'(i-k), & k < i \leq k+k'
\end{array}
\right. .
\]
Similarly, we define the skew sum of the two permutations, denoted $\pi\ominus\pi'$ and defined pointwise as
\[
\pi\ominus\pi'(i) = \left \{ 
\begin{array}{lr}
k+\pi(i), & 1\leq k,\\
\pi'(i-k), & k < i \leq k+k'
\end{array}
\right. .
\]

Let $I_k = 1,2\cdots, k$ and $J_k = k,(k-1), \cdots,1$ denote the monotone increasing and decreasing patterns, respectively.  The focus of our paper will be on permutation classes of the form $\avn{J_{k_1}\oplus I_{k_2}\oplus J_{k_3}}$ where $k_1$ and $k_3$ are positive integers and $k_2$ is a nonnegative integer.  Setting $k_1 =1$ gives the permutation class $\avn{I_{k_2+1}\oplus J_{k_3}}$ while setting $k_3 = 1$ gives $\avn{J_{k_2} \oplus I_{k_2 +1}}.$

A permuton, $\mu$, is a probability measure on the unit square with uniform marginals, that is, for $0\leq a\leq b\leq 1$,
\[
   \mu([a,b]\times[0,1]) = \mu([0,1]\times[a,b]) = b-a.
\]
We let $\mathcal{M}$ denote the space of permutons equipped with the metric \[d_{TV}(\mu,\nu) = \sup_{E\in \mathcal{B}}|\mu(E)-\nu(E)|,\] where $\mathcal{B}$ denotes Borel measurable subsets of $[0,1]^2$.  

Given a permutation $\sigma_n$ of size $n$, we define the associated permuton $\mu_{\sigma_n}$ on $E\in \mathcal{B}$ by
\[
\mu_{\sigma_{n}}(E) = n\sum_{i=1}^n\text{Leb}([(i-1)/n,i/n]\times[(\sigma_n(i)-1)/n,\sigma_n(i)/n]\cap E),
\]
where $\text{Leb}$ denotes two-dimensional Lebesgue measure.  A sequence of deterministic permutations $\{\sigma_n\}$is said to have a permuton limit, $\mu$, if the corresponding permutons $\mu_{\sigma_n}$ converge weakly to $\mu$ under the topology defined by $d_{TV}$.  If $\{\sigma_n\}$ is a sequence of random permutations we say that $\sigma_n$ converges in distribution to a random permuton $\mu$ if $\mu_{\sigma_n}$ converges in distribution to the permuton $\mu$ with respect to the topology defined above.

We will describe classes of permutations whose corresponding permuton limits converge to the unique permuton whose support is given by the line $x+y = 1$.  We denote this permuton by $\muJ$ and informally call it the anti-diagonal.  Our results readily translate to a few other classes whose permuton limit is the unique permuton, denoted $\muI$, supported on the line $x-y = 0$.


\begin{theorem}
\label{mainthm}
Fix nonnegative integers $k_1,k_2,k_3$, with $k_1$ and $k_3$ both strictly positive.  For $\pi_n$ chosen uniformly at random from $\avn{J_{k_1}\oplus I_{k_2}\oplus J_{k_3}}$, the corresponding permuton $\mu_{\pi_n}$ satisfies \[\lim_{n\to \infty} \mu_{\pi_n}  =_d  \muJ\]
where the convergence is in distribution with respect to the topology of weak convergence.  
\end{theorem}

This theorem provides partial progress to the conjecture found on page 282 of \cite{diagconj} which states that any permutation class of the form $\avn{\alpha \oplus \beta}$ will have a permuton limit of $\muJ$.  Our proof relies strongly on a bijection with $\avn{I_{k_1+k_2+k_3}}$.  We note that classes of the form $\avn{I_{k_1} \oplus J_{k_2} \oplus I_{k_3}}$ are not necessarily in bijection with $\avn{I_{k_1+k_2+k_3}}$ and are much harder to analyze.  In particular, this collection of classes contains $\avn{1324}$ which is notoriously difficult to study.

We first define some more useful notation and then give an outline of the proof.    

We often wish to translate between continuous domains and discrete domains.  For $0\leq a <c\leq n$ and $0\leq b < d \leq n $, let \[R((a,b), (c,d))= \{(x,y)\in \mathbb{R}^2: a \leq x < c, b \leq y < d\}\] denote the axes-aligned continuous rectangle with corners at $(a,b)$ and $(c,d)$, with $R(c,d) = R((0,0),(c,d))$ for short.  We note that the point $(c,d)$ is not contained in $R(c,d)$.  The $i^{th}$ column of $R(n,n)$ is the rectangle $R((i-1,0),(i,n))$ and the $j^{th}$ row of $R(n,n)$ is the rectangle $R((0,j-1),(n,j)$.  The $(i,j)^{th}$ box in $R(n,n)$ is the intersection of the $i^{th}$ column and the $j^{th}$ row.  For a subset of points with nonnegative coordinates, $A= (i_k,j_k)_{k=1}^m$, we let \[
\souw(A) = \bigcup_{k=1}^m R(i_k,j_k). 
\] 
For a permutation $\sigma \in S_n$ and a subset of indices $C\subset [n]$ define
\[
\souw(C,\sigma) = \bigcup_{i \in C} R(i,\sigma(i)).
\]
If $i_k$ is increasing, and $j_k$ is decreasing, the set $\souw(A)$ will be called a Young diagram.  
We often view a permutation $\sigma$ of size $n$ as a subset of $R(n,n)$ by filling the boxes $R((i-1,\sigma(i)-1),(i,\sigma(i))$ within the rectangle $R(n,n)$.  
For $\epsilon> 0$ and $n\in \mathbb{N}$ we define three subregions of $R(n,n)$: \[
W^+_{\epsilon,n} := \{ (x,y) \in R(n,n): x+ y > (1+ \epsilon) n \},
\]
\[
W^-_{\epsilon,n} := \{ (x,y) \in R(n,n): x+ y < (1 - \epsilon) n \}.
\]
and
\[
W_{\epsilon,n} := W^+_{\epsilon,n} \bigcup W^-_{\epsilon,n}
\]
Given a subregion $X_n\subset R(n,n)$, we let $\frac{1}{n} X_n = \{(x/n,y/n): (x,y)\in X_n \}$ denote the region of $[0,1]^2$ that is similar in shape, relative location, and orientation to $X_n$.
In particular, we let $W^+_\epsilon = \frac{1}{n}W^+_{\epsilon,n}$ and define $W^-_\epsilon$ and $W_\epsilon$ analogously.

\subsection*{Sketch of Proof}
The proof of our main theorem proceeds as follows.  In Section \ref{biject} we discuss useful properties of a type of bijection proven in Theorem 2.1 in \cite{BWX} which establishes that for any pattern $\tau$, $\avn{I_k \oplus \tau}$ is in bijection with $\avn{J_k \oplus \tau}$.  Starting from $\sigma \in \avn{I_{k_1} \oplus I_{k_2} \oplus I_{k_3}}$, using either a Backelin-West-Xin bijection or the reverse complement, we get a sequence of permutations 
\[
\sigma \to \rho \to \rho^{rc} \to \pi^{rc} \to \pi,
\]
where 
\[
\rho \in \avn{J_{k_1} \oplus I_{k_2} \oplus I_{k_3}}, \qquad \rho^{rc} \in \avn{ I_{k_3} \oplus I_{k_2} \oplus J_{k_1}},
\]
\[
\pi^{rc} \in \avn{J_{k_3} \oplus I_{k_2} \oplus J_{k_1}}, \qquad \pi \in \avn{J_{k_1} \oplus I_{k_2} \oplus J_{k_3}}.
\]

For the bijection sending $\sigma$ to $\rho$ we say there is a \emph{frozen region} in $R(n,n)$ where $\rho(i) =j$ if and only if $\sigma(i) = j$ for boxes $(i,j)$ in this region.  If $X$ is contained in this frozen region then $\mu_{\sigma}\left(\frac{1}{n} X\right) = \mu_{\rho}\left(\frac{1}{n}X\right)$.  The bijection from $\rho^{rc}$ to $\pi^{rc}$ also has a similar frozen region.  

In Section \ref{onesie} we prove some general lemmas about permutons.  Lemma \ref{oneside} says that if $\mu(W^+_\delta) < \delta$ then for some appropriate $\epsilon$, $\mu(W_\epsilon) < \epsilon$.  Lemma \ref{weakconv} show that shows if $\mu_n$ is a sequence of permutons such that for any $\epsilon> 0$, $\mu_n(W_\epsilon) < \epsilon$ for $n$ large enough, the $\mu_n$ converges weakly to $\muJ.$  

In Section \ref{structure} we explore the results in \cite{hrs4}, where it is proven, among other results, that the points of $\sigma$ can be partitioned into a collection of $d = k_1+k_2 + k_3 -1$ decreasing sequences which are not too far from the line $x+y = n+1$ and which satisfy certain use properties with high probability.  We make explicit these properties by defining a set $\avne{I_d}{\epsilon}$ and show that $\sigma$ chosen uniformly from $\avn{I_d}$ lies in $\avne{I_d}{\epsilon}$ with probability tending to 1.  Lemmas \ref{sequence} combines with \ref{epi} to show that if $\epsilon_1 <\epsilon_2 < \epsilon$ satisfy certain conditions and $\sigma \in \avne{I_d}{\epsilon_1}$, the image $\pi\in \avn{J_{k_1}\oplus I_{k_2} \oplus J_{k_3}}$ with have a permuton that satisfies $\mu_{\pi}(W_\epsilon) < \epsilon)$.  

This is seen by first showing the frozen region under the bijection from $\sigma$ to $\rho$ contains the region $W^+_{n,2\epsilon_1}$ and the corresponding permutons satisfy $\mu_\sigma(W^+_{2\epsilon_1}) = \mu_{\rho}(W^+_{2\epsilon_1}) =0$ and thus $\mu_\rho(W_{\epsilon_2}) < \epsilon_2$.  Then through a slightly more complicated argument we show that the frozen region under the bijection from $\rho^{rc}$ to $\pi^{rc}$ contains the region $W^+_{n,2\epsilon_2}$ with $\mu_{\rho^{rc}}(W^+_{2\epsilon_2}) = \mu_{\pi^{rc}}(W^+_{2\epsilon_2}) < 2\epsilon_2.$  Another application of Lemma \ref{oneside} arrives at the above conclusion.  

One difficultly of the proof is finding a way to show that the frozen region under the bijection that maps $\rho^{rc}$ to $\pi^{rc}$ is sufficiently large enough to contain a region of the form $W^{+}_{n,2\epsilon_2}$.  If the frozen region under this bijection were too small, our proof would fall apart.  However, using results developed in \cite{hrs4}, Lemmas \ref{goodprops} and \ref{sequence} provides a set of conditions that hold with high probability.  The proof of Lemma \ref{epi} uses those properties to show that frozen region will, in fact, be large enough.

We can then conclude that for a sequence of permutations $\{\pi_n\}$ with $\pi_n \in \avn{J_{k_1} \oplus I_{k_2} \oplus J_{k_3}}$ chosen uniformly at random, the permuton $\mu_{\pi_n}$ will have $\mathbb{P}_n(\mu_{\pi_n}(W_\epsilon) < \epsilon) \to 1$, and thus $\mu_{\pi_n} \to_d \muJ$ where the convergence is in distribution with respect to topology of weak convergence.

\section{Backelin-West-Xin Bijection}\label{biject}
In this section, we discuss the bijection from \cite{BWX}, with some adjustments to their notation as shown in \cite{BaxterJaggard}. First, we define the notions of Wilf-equivalence and shape-Wilf equivalence. Let $S_n(\sigma)$ denote the set of permutations of size $n$ that do not contain an occurrence of a permutation $\sigma$. We say that two permutations $\sigma$ and $\sigma'$ are \textit{Wilf-equivalent}, denoted $\sigma \sim_W \sigma'$, if $|S_n(\sigma)| = |S_n(\sigma')|$ for all $n$. Hence, under this definition all length 3-avoiding permutations are Wilf equivalent, since for all $\tau \in S_3$, we have $|S_n(\tau)| = C_n$, the $n$th Catalan number.

\begin{figure}[htbp]
     \centering
    \begin{tikzpicture}[scale=0.35]
       \draw [lightgray] (0,0) grid (5,6);
       \draw [lightgray] (0,5) grid (4,7);
       \draw [lightgray] (5,0) grid (7,3);


       \fill[fill=blue, fill opacity=0.4] (3,5) rectangle (4,7);
       \fill[fill=blue, fill opacity=0.4] (4,5) rectangle (5,6);
      
       \filldraw[black] (.5,0.5) circle (5pt);
       \filldraw[black] (1.5,3.5) circle (5pt);
       \filldraw[black] (2.5,4.5) circle (5pt);
       \filldraw[black] (3.5,6.5) circle (5pt);
       \filldraw[black] (4.5,5.5) circle (5pt);
       \filldraw[black] (5.5,1.5) circle (5pt);
       \filldraw[black] (6.5,2.5) circle (5pt);
       
     \end{tikzpicture}    
     \hspace{2cm} 
     \begin{tikzpicture}[scale=0.35]
       \draw [lightgray] (0,0) grid (5,6);
       \draw [lightgray] (0,5) grid (4,7);
       \draw [lightgray] (5,0) grid (7,3);


       \fill[fill=blue, fill opacity=0.4] (0,4) rectangle (2,7);
       \fill[fill=brown, fill opacity=0.4] (1,3) rectangle (3,5);
      
       \filldraw[black] (.5,6.5) circle (5pt);
       \filldraw[black] (1.5,4.5) circle (5pt);
       \filldraw[black] (2.5,3.5) circle (5pt);
       \filldraw[black] (3.5,5.5) circle (5pt);
       \filldraw[black] (4.5,0.5) circle (5pt);
       \filldraw[black] (5.5,1.5) circle (5pt);
       \filldraw[black] (6.5,2.5) circle (5pt);
       
     \end{tikzpicture}    
     \caption{The traversal on the left avoids the pattern $21$ as any two points in decreasing order (like those highlighted) is not contained in a rectangle within the bounding Young diagram.  The traversal on the right contains copies of the pattern $21$ (some of which are highlighted).}
     \label{fi:sweex}
\end{figure}

For shape-Wilf equivalence, consider a Young diagram $\lambda$, where $\lambda = (\lambda_1, \lambda_2, \dots, \lambda_n)$.  We use the French convention for Young diagrams so the $i$th row starting from the bottom row has length $\lambda_i$.  We define $L$ a \textit{traversal} of $\lambda$ as a filling of boxes of $\lambda$ with zeros and ones, with exactly one 1 in each row and column; for our purposes, we will illustrate ones with dots, corresponding to entries of a permutation, and zeros with empty boxes. Given $\sigma \in S_m$, let $M_\sigma$ denote the traversal given by $\sigma$ of the $m \times m$ Young diagram. We say that $L$ \textit{contains} a copy of $\sigma$ if $M_\sigma$ (including all its $1$s and $0$s) is present in $L$ when restricting to rows $r_1 < r_2 < \dots < r_m$ and columns $c_1 < c_2 < \dots < c_m$. Hence, we view the pattern as being present only if it is contained in a complete rectangle contained within the Young diagram (See Figure \ref{fi:sweex}). On the other hand, we say that $L$ \textit{avoids} $\sigma$ if $L$ contains no copy of $\sigma$. Let $S_\lambda (\sigma)$ denote the traversals of $\lambda$ that avoid $\sigma$. Two permutations $\sigma$ and $\sigma'$ are \textit{shape-Wilf equivalent}, denoted $\sigma \sim_{sW} \sigma'$, if $|S_\lambda (\sigma)| = |S_\lambda (\sigma')|$ for all Young diagrams $\lambda$. For example, the shape-Wilf equivalence classes for length 3 permutations are $\{213, 132 \}$, $\{123, 231, 321\}$, and $\{312\}$.


A result from \cite{BWX} relates to the shape-Wilf equivalence of larger permutations that contain shape-Wilf equivalent patterns. We adapt to the notation shown in \cite{BaxterJaggard}.

\begin{proposition}[\cite{BWX}]{\label{genbwx}}
    If $\sigma \sim_{sW} \sigma'$, then $\sigma \oplus \tau \sim_{sW} \sigma' \oplus \tau$.
\end{proposition}

The proof of this proposition implies a bijection between $\avn{\sigma\oplus\tau}$ and $\avn{\sigma'\oplus\tau}$ that fixes a large set of points for many of the permutations. This process is illustrated in Figure \ref{fi:bwx} with a length 20 permutation.  Fix $\pi \in \avn{\sigma\oplus \tau}$; we color the box $(i,j)\in [n]^2$ white if there is no occurrence of $\tau$ using points $(i',\pi(i'))$ where $i'>i$ and $\pi(i')>j$. Otherwise we color the box blue.
%
Let $\lambda = \lambda(\pi,\tau)$ denote the Young diagram obtained by deleting all white boxes and rows and columns that contain a point $(i,\pi(i))$ in a white box. 
The points $(i,\pi(i))$ that are colored blue form a traversal of $\lambda$, which we denote by $\pi_\lambda$. Now, $\pi_\lambda \in S_\lambda(\sigma)$. Since $\sigma' \sim_{sW} \sigma$, there exists a bijection from $S_\lambda(\sigma)$ to $S_\lambda(\sigma')$. Let $\pi'_\lambda$ denote the image under this bijection of $\pi_\lambda$. This induces a bijection between $\avn{\sigma \oplus \tau}$ and $\avn{\sigma' \oplus \tau}$.  As only the points in blue boxes are changed under this bijection, we call the white boxes the \emph{frozen region} under this bijection.

\begin{figure}[thp]
\centering
    \begin{tikzpicture}[scale=0.2]
       \draw [lightgray] (0,0) grid (20,20);
    
       \filldraw[black] (.5,13.5) circle (5pt);
       \filldraw[black] (1.5,9.5) circle (5pt);
       \filldraw[black] (2.5,16.5) circle (5pt);
       \filldraw[black] (3.5,7.5) circle (5pt);
       \filldraw[black] (4.5,19.5) circle (5pt);
       \filldraw[black] (5.5,5.5) circle (5pt);
       \filldraw[black] (6.5,14.5) circle (5pt);
       \filldraw[black] (7.5,2.5) circle (5pt);
       \filldraw[black] (8.5,12.5) circle (5pt);
       \filldraw[black] (9.5,18.5) circle (5pt);
       \filldraw[black] (10.5,10.5) circle (5pt);
       \filldraw[black] (11.5,8.5) circle (5pt);
       \filldraw[black] (12.5,1.5) circle (5pt);
       \filldraw[black] (13.5,0.5) circle (5pt);
       \filldraw[black] (14.5,17.5) circle (5pt);
       \filldraw[black] (15.5,15.5) circle (5pt);
       \filldraw[black] (16.5,3.5) circle (5pt);
       \filldraw[black] (17.5,11.5) circle (5pt);
       \filldraw[black] (18.5,6.5) circle (5pt);
       \filldraw[black] (19.5,4.5) circle (5pt);
    
        \fill[fill=blue, fill opacity=0.4] (2,12) rectangle (0,16);
        \fill[fill=blue, fill opacity=0.4] (5,12) rectangle (2,14);
        \fill[fill=blue, fill opacity=0.4] (8,12) rectangle (0,0);
        \fill[fill=blue, fill opacity=0.4] (8,10) rectangle (10,0);
        \fill[fill=blue, fill opacity=0.4] (11,8) rectangle (10,0);
        \fill[fill=blue, fill opacity=0.4] (5,12) rectangle (6,14);
    
        \fill[fill=blue, fill opacity=0.4] (11,3) rectangle (16,0);
       
     \end{tikzpicture}    
     \qquad 
    \begin{tikzpicture}[scale=0.2]
         \draw [lightgray] (0,0) grid (20,20);
    
       \filldraw[black] (.5,13.5) circle (5pt);
       \filldraw[black] (1.5,9.5) circle (5pt);
       \filldraw[black] (2.5,16.5) circle (5pt);
       \filldraw[black] (3.5,7.5) circle (5pt);
       \filldraw[black] (4.5,19.5) circle (5pt);
       \filldraw[black] (5.5,5.5) circle (5pt);
       \filldraw[black] (6.5,14.5) circle (5pt);
       \filldraw[black] (7.5,2.5) circle (5pt);
       \filldraw[black] (8.5,12.5) circle (5pt);
       \filldraw[black] (9.5,18.5) circle (5pt);
       \filldraw[black] (10.5,10.5) circle (5pt);
       \filldraw[black] (11.5,8.5) circle (5pt);
       \filldraw[black] (12.5,1.5) circle (5pt);
       \filldraw[black] (13.5,0.5) circle (5pt);
       \filldraw[black] (14.5,17.5) circle (5pt);
       \filldraw[black] (15.5,15.5) circle (5pt);
       \filldraw[black] (16.5,3.5) circle (5pt);
       \filldraw[black] (17.5,11.5) circle (5pt);
       \filldraw[black] (18.5,6.5) circle (5pt);
       \filldraw[black] (19.5,4.5) circle (5pt);
    
        \fill[fill=blue, fill opacity=0.4] (2,12) rectangle (0,16);
        \fill[fill=blue, fill opacity=0.4] (5,12) rectangle (2,14);
        \fill[fill=blue, fill opacity=0.4] (8,12) rectangle (0,0);
        \fill[fill=blue, fill opacity=0.4] (8,10) rectangle (10,0);
        \fill[fill=blue, fill opacity=0.4] (11,8) rectangle (10,0);
        \fill[fill=blue, fill opacity=0.4] (5,12) rectangle (6,14);
        \fill[fill=blue, fill opacity=0.4] (11,3) rectangle (16,0);
    
        \begin{scope}[transparency group, fill opacity=0.4]
        \fill[fill=brown] (3,0) rectangle (2,14);
        \fill[fill=brown] (5,0) rectangle (4,14);
        \fill[fill=brown] (7,0) rectangle (6,12);
        \fill[fill=brown] (9,0) rectangle (8,10);
        \fill[fill=brown] (10,0) rectangle (9,10);
        \fill[fill=brown] (11,0) rectangle (10,8);
        \fill[fill=brown] (12,0) rectangle (11,3);
        \fill[fill=brown] (16,0) rectangle (14,3);
    
        \fill[fill=brown] (0,14) rectangle (2,16);
        \fill[fill=brown] (0,12) rectangle (6,13);
        \fill[fill=brown] (0,10) rectangle (8,12);
        \fill[fill=brown] (0,8) rectangle (9,9);
        \fill[fill=brown] (0,6) rectangle (10,7);
        \fill[fill=brown] (0,3) rectangle (11,5);
        \end{scope}
        
     \end{tikzpicture}    
\qquad
    \begin{tikzpicture}[scale=0.2]
            \draw [lightgray] (0,0) grid (20,20);
        
           \filldraw[black] (.5,0.5) circle (5pt);
           \filldraw[black] (1.5,5.5) circle (5pt);
           \filldraw[black] (2.5,16.5) circle (5pt);
           \filldraw[black] (3.5,7.5) circle (5pt);
           \filldraw[black] (4.5,19.5) circle (5pt);
           \filldraw[black] (5.5,13.5) circle (5pt);
           \filldraw[black] (6.5,14.5) circle (5pt);
           \filldraw[black] (7.5,9.5) circle (5pt);
           \filldraw[black] (8.5,12.5) circle (5pt);
           \filldraw[black] (9.5,18.5) circle (5pt);
           \filldraw[black] (10.5,10.5) circle (5pt);
           \filldraw[black] (11.5,8.5) circle (5pt);
           \filldraw[black] (12.5,1.5) circle (5pt);
           \filldraw[black] (13.5,2.5) circle (5pt);
           \filldraw[black] (14.5,17.5) circle (5pt);
           \filldraw[black] (15.5,15.5) circle (5pt);
           \filldraw[black] (16.5,3.5) circle (5pt);
           \filldraw[black] (17.5,11.5) circle (5pt);
           \filldraw[black] (18.5,6.5) circle (5pt);
           \filldraw[black] (19.5,4.5) circle (5pt);
        
            \fill[fill=blue, fill opacity=0.4] (2,12) rectangle (0,16);
            \fill[fill=blue, fill opacity=0.4] (5,12) rectangle (2,14);
            \fill[fill=blue, fill opacity=0.4] (8,12) rectangle (0,0);
            \fill[fill=blue, fill opacity=0.4] (8,10) rectangle (10,0);
            \fill[fill=blue, fill opacity=0.4] (11,8) rectangle (10,0);
            \fill[fill=blue, fill opacity=0.4] (5,12) rectangle (6,14);
            \fill[fill=blue, fill opacity=0.4] (11,3) rectangle (16,0);
        
        \begin{scope}[transparency group, fill opacity=0.4]
        \fill[fill=brown] (3,0) rectangle (2,14);
        \fill[fill=brown] (5,0) rectangle (4,14);
        \fill[fill=brown] (7,0) rectangle (6,12);
        \fill[fill=brown] (9,0) rectangle (8,10);
        \fill[fill=brown] (10,0) rectangle (9,10);
        \fill[fill=brown] (11,0) rectangle (10,8);
        \fill[fill=brown] (12,0) rectangle (11,3);
        \fill[fill=brown] (16,0) rectangle (14,3);
    
        \fill[fill=brown] (0,14) rectangle (2,16);
        \fill[fill=brown] (0,12) rectangle (6,13);
        \fill[fill=brown] (0,10) rectangle (8,12);
        \fill[fill=brown] (0,8) rectangle (9,9);
        \fill[fill=brown] (0,6) rectangle (10,7);
        \fill[fill=brown] (0,3) rectangle (11,5);
        \end{scope}
        \end{tikzpicture}
\caption{Illustration of coloring procedure in proof of Proposition \ref{genbwx} and the Backelin-West-Xin bijection with $\sigma = I_2,$ $\sigma' = J_2$ and $\tau = I_2$.  The white region of the square is called the \emph{frozen region}.}
\label{fi:bwx}
\end{figure}

Applying the Backelin-West-Xin bijection, for the example shown in Figure \ref{fi:bwx}, the traversal $\pi'_\lambda$ avoids the pattern $21$ (by the previous discussion of pattern containment for Young diagrams).

Using the shape-Wilf-equivalence of $J_k$ and $J_{k-m}\oplus I_m$ for any $0\leq m<k$ that was proven in \cite{BWX}, we obtain a particularly useful bijection, summarized in Proposition \ref{bwx}. 

\begin{proposition}[Backelin-West-Xin Bijection]\label{bwx}
   For any fixed pattern $\sigma$, with $\sigma \in S_{m}$, there exists a bijection between the sets of pattern-avoiding permutations $\avn {I_{k} \oplus \sigma}$ and $\avn{J_k \oplus \sigma}$. 
\end{proposition}

\section{Analytical Tools}\label{onesie}
The following lemma will be a useful tool in showing permuton collapse to the anti-diagonal.  

\begin{lemma}[One-sided lemma]\label{oneside}
    Fix small positive $\delta\in (0,.5)$.  If $\mu$ is a permuton such that $\mu(\utri_{\delta}) < \delta$, then for $\epsilon > 2/\log(\delta^{-1}) > \delta$, $\mu(\ltri_\epsilon) < \epsilon$ and $\mu(W_{2\epsilon})< 2\epsilon$.
\end{lemma}

\begin{proof}
    For simplicity, we let $\epsilon = 1/m$ where $m < \log(\delta^{-1})/2$.  For $0 \leq i \leq m-2$, consider the two strips $V_i = [1-(i+1)\epsilon,1-i\epsilon)\times [0,1)$ and $H_i= [0,1)\times [i\epsilon,(i+1)\epsilon + \delta)$. These satisfy $\mu(V_i) = \epsilon,$ $\mu(H_i) = \epsilon+\delta$.  We partition these two sets into a union of disjoint sets, $V_i = V_i^+\cup V_i^\circ \cup V_i^-$ and $H_i = H_i^+\cup H_i^\circ \cup H_i ^-$, as shown in Figure \ref{hsandvs}, where: 
   
 \begin{figure}[htbp]
 \centering
 	\includegraphics[scale=.5]{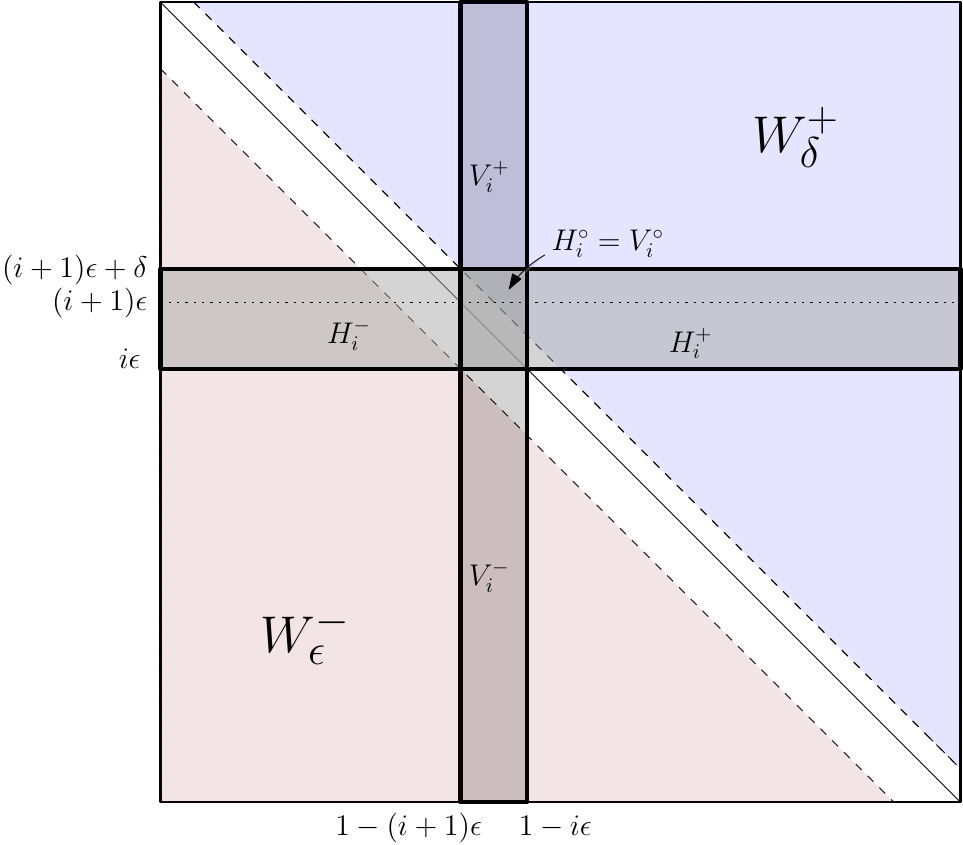}
    \caption{An example of a partition of $H_i = H_i^-\bigcup H_i^\circ\bigcup H_i^+$ and $V_i = V_i^+ \bigcup V_i^\circ \bigcup V_i^-$. The blue shaded region represents $\utri_\delta$ while the pink shaded region is $\ltri_\epsilon$.  Note that $V_i^+ \subset \utri_\delta$, $V_i^- \subset \bigcup_{j=i+1}^{m-2} H^-_j$, and $\ltri_\epsilon \subset \bigcup_{i=0}^{m-2} H_i^-$. }
    \label{hsandvs}
\end{figure}   
    
    \begin{itemize}
    \item $V_i^\circ = H_i^\circ = V_i \cap H_i$,  
    \item $V_i^+  =[1-(i+1)\epsilon,1-i\epsilon)\times[(i+1)\epsilon+\delta, 1)$,
    \item $V_i^-  =[1-(i+1)\epsilon,1-i\epsilon)\times[0,i\epsilon)$,
    \item $H_i^+ = [1-i\epsilon, 1) \times [i\epsilon,(i+1)\epsilon+\delta)$,
    \item $H_i^- = [0,1-(i+1)\epsilon) \times [i\epsilon,(i+1)\epsilon+\delta)$.
\end{itemize}

The set $\ltri_\epsilon$ is a subset of $\bigcup_{0\leq i \leq m-2}H_i^-.$  Thus the proof will follow by showing that \[\sum_{i = 0}^{m-2} \mu(H_i^-) < \epsilon.\]   
We claim that $\mu(H_i^+)< 2^{i+1}\delta$.  The proof of this claim proceeds via induction.  For $0\leq i \leq m-2$ we have 
\[
\mu(H_i) - \mu(V_i) = \mu(H_i^-) + \mu(H_i^\circ) + \mu(H_i^+) - \mu(V_i^-) - \mu(V_i^\circ)- \mu(V_i^+) = \delta.
\]
Noting that $\mu(H_i^\circ)$ and $\mu(V_i^\circ)$ cancel each other out and $\mu(V_i^+) < \delta$, we have by rearranging terms
\begin{align*}
\mu(H_i^-) &= \delta + \mu(V_i^-) + \mu(V_i^+) - \mu(H_i^+) \\
& \leq \delta + \mu(V_i^-) + \delta\\
&= 2\delta + \mu(V_i^-).
\end{align*}

For $i = 0$, we note that $V_0^- = H_0^+ = \emptyset$ and thus
\[
    \mu(H_0^-) < 2\delta,
\]
establishing the base case.  For $1\leq i \leq m-2$, suppose the claim holds for $H^-_j$ where $j< i$.  We observe that $V_i^- \subset \bigcup_{0\leq j <i} H^-_j$ and therefore $\mu(V_i^-) \leq \sum_{j=0}^{i-1} \mu(H^-_j)$.  Then by the induction hypothesis we have
\[
\mu(H_i^-) < 2\delta + \sum_{j=0}^{i-1} 2^{j+1}\delta = 2^{i+1}\delta,
\]
and thus the claim holds for all $0 \leq i \leq m-2$.  Then,
\[
\mu(\ltri_\epsilon) \leq \sum_{i=0}^{m-1} \mu(H_i^-) < \sum_{i=0}^{m-2} 2^{i+1}\delta < 2^m\delta.
\] 
If $\epsilon > 2/\log(1/\delta)$, then $m< \log(1/\delta)$ and the above bound becomes
\[
\mu(\ltri_\epsilon)< \delta^{1-\log 2}.
\]
For $0< \delta<1$ we have $\delta^{1-\log(2)}< 2/\log(1/\delta) < \epsilon$, and thus we may conclude 
$\mu(\ltri_\epsilon)< \epsilon$.  For $\delta<.5$ and $\epsilon > 2/\log(1/\delta)$ we have that $\delta< \epsilon$.  Furthermore, \[\mu(W_{2\epsilon}) \leq \mu(W_\epsilon) \leq \mu( W^-_{\epsilon}) + \mu(W^+_{\delta}) < 2\epsilon.\]  We note that as $\delta$ approaches $0$, $2/\log(1/\delta)$ also approaches $0$, and thus $\epsilon$ can be made arbitrarily small by choosing sufficiently small $\delta.$
\end{proof}


\begin{lemma}\label{weakconv}
    Let $\mu_n$ be a sequence of permutons where for any $\epsilon >0$, there exists $N$ large enough such that for $n > N$, $\mu_n(W_\epsilon) < \epsilon$.  Then $\mu_n$ converges weakly to $\muJ$.    
\end{lemma}
\begin{proof}

The set of rectangles $\mathcal{R}$ form a basis for the Borel sets of $[0,1]^2$ and thus by the Portmanteau theorem, weak convergence follows if $\sup_{R\in\mathcal{R}}|\mu_n(R) - \muJ(R)| \to 0$. 

Fix $\epsilon > 0$.  Let $n$ be large enough such that $\mu_n(W_\epsilon)< \epsilon$ and let $R \in \mathcal{R}$.   We denote the smallest axis aligned square that contains the intersection of $R$ with the anti-diagonal by $S\subset R$ so that $\mu_J(R) = \mu_J(S)$.  Let $X = [a,b]$ be the projection of $S$ on the $x-$axis.  Then $\mu_J(R) = \mu_J(S) = b-a$.  Let $X^+_\epsilon = [a-\epsilon,b+\epsilon]$ and $X^-_\epsilon = [a+\epsilon, b-\epsilon]$.  If $\mu_J(S) = b-a < 2\epsilon$ then $X^-_\epsilon = \emptyset.$ Finally let $C^{\pm}_\epsilon = X^{\pm}_\epsilon \times [0,1]$.  With these definitions we have that $R\subset C^+_\epsilon \cup W_\epsilon$ and thus 
\[
\mu_n(R) \leq \mu_n(C^+_\epsilon) + \mu_n(W_\epsilon) < (b-a) + 2\epsilon + \epsilon.
\]
On the other hand we also have that $C^-_\epsilon\backslash R \subset W_\epsilon$ and thus
\[
\mu_n(R) > \mu_n(C^-_\epsilon) - \mu_n(W_\epsilon) > (b-a) - 2\epsilon - \epsilon.
\]
These two inequalities combine to show that $|\mu_n(R) - \mu_J(R)| < 3\epsilon.$  Thus $\sup_{R\in \mathcal{R}} |\mu_n(R) - \muJ(R) | \to 0$ as $\epsilon$ can be chosen to be arbitrarily small. 
\end{proof}

\section{Structure of permutations}\label{structure}

We describe a canonical partition of $[n]$ with respect to a permutation $\sigma \in \avn{I_{d+1}}$ in order to keep track of how points of $\sigma$ in each part are moved around under a sequence of bijections.  This partition and related properties follow by translating results in \cite{hrs4} for $\tau \in \avn{J_{d+1}}$ and looking at $\sigma = \tau^c$.

For a permutation $\sigma\in \avn{I_{d+1}}$ of size $n$, and a subset of the indices $C\subset [n]$, we say $i\in C$ is a left-to-right minimum of $C$ with respect to $\sigma$ if for $j \in C$ with $j<i$, $\sigma(j) > \sigma(i)$.  Let $A^1_\sigma\subset [n]$ denote the left-to-right minima of $[n]$ with respect to $\sigma$. Then define recursively for $l>1$, $A^l_\sigma$ to be the left-to-right minima of $[n] \setminus \bigcup_{i = 1}^{l-1} A^i$ with respect to $\sigma$.  This partitions $[n]$ into $d$ disjoint (possibly empty) subsets.  This partition has the property that for $l>1$ and $i\in A^l_\sigma$, there exists a sequence of indices $i_1 < \cdots < i_l = i$ such that $i_j \in A^j$ for $1\leq j \leq l$ and $\sigma(i_1)< \cdots < \sigma(i_l)$.  This canonical partition does not guarantee the existence of a complete increasing sequence $i_1 < \cdots< i_l= i < \cdots < i_d $ with $i_j \in A^j$ for $1\leq j \leq d$.  However, we will show that for most permutations and most indices in $[n]$ such a complete increasing sequence does exist.


We obtain a collection of $d$ continuous functions $P_\sigma = (f^1_\sigma,\cdots f^d_\sigma)$ by linearly interpolating between the points $(0,0)$, $(1,0)$ and \[\left (\frac{i}{n+1},\frac{\sigma(i) + i- (n+1)}{\sqrt{2dn}}\right )_{i \in A^l_{\sigma}}\] 
for each $l\in [d]$.  
We translate results from \cite{hrs4} for $\sigma^* \in \avn{J_d}$ to our situation via the bijection $\sigma = \sigma^{*c}$. Translating Theorem 2.1 in \cite{hrs4} shows that this collection of functions converges in distribution with respect to the sup norm topology to $\Lambda = (\lambda^1,\cdots, \lambda^d)$ where the $\lambda^l$'s are each a Brownian bridge from $(0,0)$ to $(1,0)$ jointly conditioned to satisfy:
\begin{enumerate}
    \item $\lambda_l(0) = \lambda_l(1) = 0$ for all $l\in d$,
    \item $\lambda_1(t) < \lambda_2(t) < \cdots < \lambda_d(t)$, for $t \in (0,1)$
    \item $\sum_{l=1}^d \lambda_l(t) = 0.$
\end{enumerate}
This object is called a traceless Dyson Brownian bridge.  See the appendix in \cite{hrs4} for details relating the traceless Dyson Brownian bridge to the collection of $d$ Brownian bridges conditioned to satisfy the above properties. 


We define $j^l(i) := \max\{j< i: j \in A^l_\sigma\}$ unless $A^l_\sigma$ does not contain any points less than $i$, in which case $j^l(i)$ takes on the default value of $0$.  For $\epsilon >0$, let $\avne{I_d}{\epsilon}$ denote the subset of permutations in $\avn{I_d}$ that satisfy the following properties:
\begin{enumerate}
    \item $\sup_{i \in [n] }| \sigma(i) + i-n-1| < n^{.6}$.
    \item For $\epsilon n \leq i < j \leq (1-\epsilon )n$ with $j-i >n^{.1}$, and for each $1\leq l \leq d$, \[\left||A^l(\sigma)\bigcap [i,j]| - d^{-1}(j-i)\right| < |j-i|^{.6}.\] 
    \item For $\epsilon n \leq i < j \leq (1-\epsilon)n$ with $j-i> n^{.1}$ and $i,j \in A^l(\sigma)$, \[|\sigma(i) +i -\sigma(j)-j| < |j-i|^{.6}.\]
    \item For any $1\leq l \leq d$, and $i \in  [\epsilon n , (1-\epsilon)n]$, 
    \(
    i-j^{l}(i)< n^{.2}. 
    \) 
    \item For $1\leq l<d$ and $i \in A^l_\sigma\cap[ \epsilon n , (1-\epsilon )n]$,
    \(
        \sigma(j^{l+1}(i)) -\sigma(i) > n^{.4}.
    \)
\end{enumerate}
\begin{lemma}\label{goodprops}
	Fix $\epsilon>0$ and let $\prob_n$ denote uniform measure on $\avn{I_{d+1}}$.  Then $\prob_n(\sigma \in \avne{I_{d+1}}{\epsilon}) \to 1.$ 
\end{lemma}

\begin{proof}

The proof will follow from results in \cite{hrs4}.  Similar to Section 2 of \cite{hrs4}, for a permutation $\sigma \in \avn{I_{d+1}}$ we define the projection to the pair of sequences $\omega_\sigma \in [d]^n\times [d]^n$ where $\omega_\sigma(i) = (l_1,l_2)$ if $i \in A^{l_1}_\sigma$ and $\sigma^{-1}(i)\in A^{l_2}_\sigma$.  Let $a\times b = \omega_\sigma$.  In Section 4 of \cite{hrs4}, the authors define a set of properties for pairs of sequences called the ``Petrov Conditions.''  The conditions translate to the first four conditions of $\avne{I_{d+1}}{\epsilon}$ via a series of Lemmas in that section.  Lemma 7.11 from \cite{hrs4} shows that for sequences that come from $\avn{I_{d+1}}$, the ``Petrov Conditions'' hold asymptotically almost surely. 


To show the last property holds for $\avn{I_{d+1},\epsilon}$, we will assume the first four are already true.  Let $X$ be the $\inf_{1\leq l <d, t\in [\epsilon,1-\epsilon]} (\lambda^{l+1}(t) - \lambda^l(t) )$.  Let $Y_n$ be similar defined for $(f^1_\sigma,\cdots, f^d_\sigma)$ where $\sigma$ is chosen uniformly form $\avn{J_d}.$  The second property in the definition of Dyson Brownian bridge along with continuity of the component paths combine to show that $\prob(X>c) \to 1$ as $c \to 0.$  Convergence in distribution of $(f^1_\sigma, \cdots , f^d_\sigma)$ to $(\lambda^1,\cdots, \lambda^d)$ implies that $Y_n \to_d X$ as well, and thus $\prob(Y_n> n^{-.05})$ tends to 1 as $n$ increases.  If $Y_n > n^{-.05}$, then $\sqrt{2dn}(f^{l+1}_\sigma(t) -f^l_\sigma(t)) > n^{.45}$.  Fix $t \in [\epsilon , 1-\epsilon]$.  Let $i = \lfloor nt \rfloor$.  Then $i$ is in $A^l_\sigma$ for some $l$.  If $l<d$, let $n(t-sn^{-.8}) = j^{l+1}(i)$.  By the second to last property of $\avne{I_{d+1}}{\epsilon}$, $s$ will lie in the interval $(0,1)$.  Using the third property we see that $\sqrt{2dn}(f^{l+1}_\sigma(t)- f^{l+1}_\sigma(t-sn^{-.8})) < n^{.2}.$  Combining everything we have with probability that tends to 1,
\begin{align*}
    \sigma(j^{l+1}(i))- \sigma(i) > & \sqrt{2dn}(f^{l+1}_\sigma((t-sn^{-.8}) - f^l_\sigma(t)) - n^{.2} \\
    >& \sqrt{2dn}(f^{l+1}_\sigma(t) - f^l_\sigma(t)) - 2n^{.2}\\
    >& n^{.45} - 2n^{.2}\\
    >& n^{.4}.
\end{align*}
\end{proof}

Lemma \ref{goodprops} implies a useful lemma that will allow us to pin down the frozen region under certain Backelin-West-Xin bijections.

\begin{lemma}\label{sequence}
Fix $\epsilon > 0$ and let $\sigma \in \avne{I_{d+1}}{\epsilon}$.  For each $1\leq l \leq d$ and each $i \in [2\epsilon n, (1-2\epsilon)n]\cap A^l_{\sigma}$, there exists a sequence $i_1 < \cdots < i_d$ with $i_k \in A^k_\sigma$ such that $i = i_l$ and $\sigma(i_1) < \cdots < \sigma(i_d)$. 
\end{lemma}
\begin{proof}
Let $i\in A^l(\sigma) \cap [2\epsilon n, (1-2\epsilon)n]$.  By the definition of $A^l(\sigma)$, there must by $i_1 < \cdots i_{l-1} < i$ that already satisfy $i_k \in A^k(\sigma).$  The key is showing that for $i \in (2\epsilon n, (1-2\epsilon)n)$, the sequence continues after $i_l = i$ all the way to $i_d$.  Combining the last two properties of $\avne{I_{d+1}}{\epsilon}$ there is a sequence $\epsilon n < i_d' < i_{d-1}' \cdots < i_{l}$ such that $i_l- i_d' < dn^{.2}$ and $\sigma(i'_k) - \sigma(i'_{k-1}) > n^{.4}$ for $l+1 \leq k \leq d.$  By the second property of $\avne{I_{d+1}}{\epsilon}$ there is a sequence $i_l< i_{l+1}<  \cdots < i_d < (1-\epsilon)n$ such that $i_d - i_{l} < dn^{.2}$ and each $i_k \in A^k(\sigma)$ for $l < k \leq d.$  Finally, $|\sigma(i_k') - \sigma(i_k)| < 2dn^{.12} + 2dn^{.2}$ by the last property of $\avne{I_{d+1}}{\epsilon}$ and thus $ \sigma(i_k) - \sigma(i_{k-1}) > n^{.4} - 8dn^{.2} > 0$, finishing the proof.  
\end{proof}



\section{Proof of Theorem \ref{mainthm}}
Letting $d+1 = k_1 + k_2 + k_3,$ we start with a permutation $\sigma \in \avn{I_{d+1}}$ (suppressing the subscript $n$ for readability).  We let $\rho\in \avn{J_{k_1} \oplus I_{k_2} \oplus I_{k_3}}$ denote the image of $\sigma$ via the Backelin-West-Xin bijection defined in Section \ref{biject}.  Then $\rho^{rc}\in \avn{I_{k_3} \oplus I_{k_2} \oplus J_{k_1}}$ and we let $\pi^{rc} \in \avn{J_{k_3} \oplus I_{k_2} \oplus J_{k_1}}$ denote the image of $\rho^{rc}$ under the Backelin-West-Xin bijection.  Finally we end up with $\pi\in \avn{J_{k_1} \oplus I_{k_2} \oplus J_{k_3}}$ by taking the reverse complement of $\pi^{rc}.$  The main idea of the proof for Theorem \ref{mainthm} comes from the following lemma.

\begin{lemma}\label{epi}
    Let $\epsilon_1 > 0$ and let $\epsilon_2 > \max( 4\epsilon_1, 2/\log(2\epsilon_1)^{-1})$,  $\epsilon> 4/\log(\epsilon_2^{-1})$.  For integers $k_1>0$, $k_2\geq 0 $, and $k_3>0.$  Let $\pi \in \avn{J_{k_1}\oplus I_{k_2} \oplus J_{k_3}}$ be the image of $\sigma \in \avne{I_{k_1} \oplus I_{k_2} \oplus I_{k_3}}{\epsilon_1}.$  Then $\mu_{\pi}(W_\epsilon) < \epsilon.$
\end{lemma}

\begin{proof}

\begin{figure}[hptb]
\centering
    \includegraphics[scale = .45]{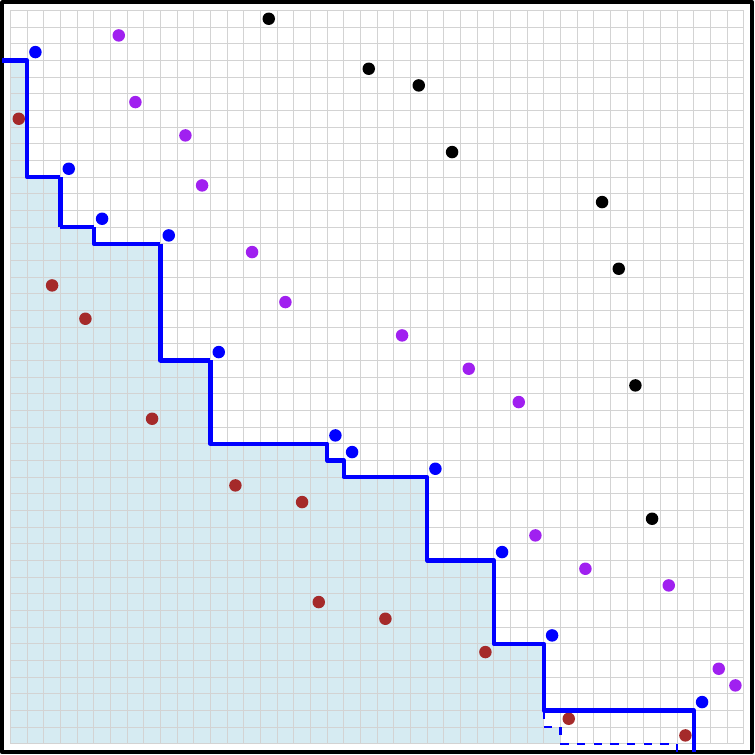} \qquad 
    \includegraphics[scale = .45]{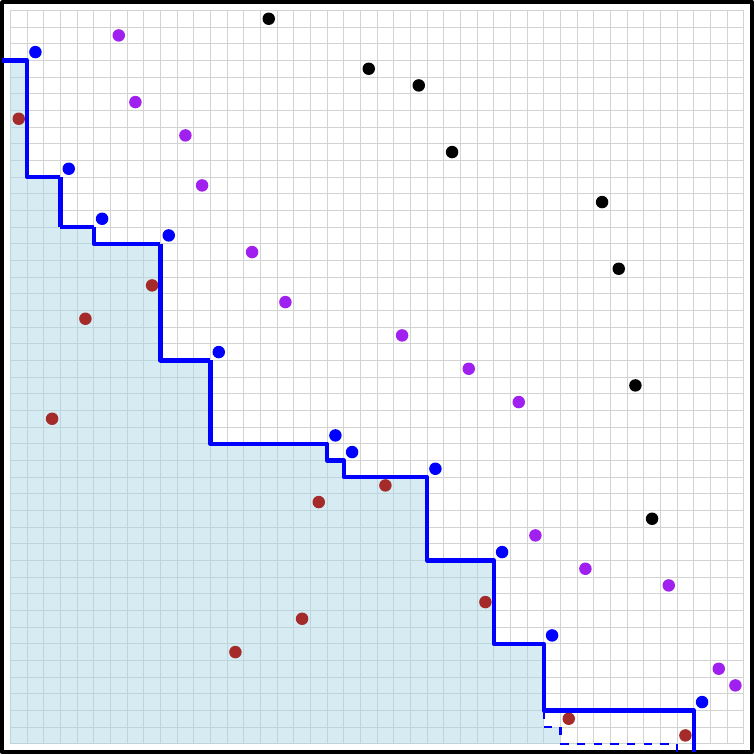}
    \caption{A permutation, $\sigma \in \avn{I_5}$, and its image under the Backelin-West-Xin bijection, $\rho \in \avn{J_2\oplus I_3}$, with region $\souw(A^2_\sigma,\sigma)$ outlined by solid blue line.  The frozen region under this bijection that is not shaded can properly contain the complement of $\souw(A^2_\sigma,\sigma)$ as the two lowest brown shaded points of $A^1_\sigma$ do not have a copy of $I_3$ northeast of those boxes.}
    \label{figAB}
\end{figure}

Let $\rho$, $\rho^{rc}$, $\pi^{rc}$, and $\pi$ be the images under the appropriate bijections defined above.  For any box not in $\souw(A^{k_1}_{\sigma},\sigma)$, there can be no copy of $I_{k_2}\oplus I_{k_3}$ in $\sigma$ northwest of that box.  Thus, the frozen region under the bijection from $\sigma$ to $\rho$ will contain the  complement of $\souw(A^{k_1}_{\sigma},\sigma).$  That is, for $(i,j)\notin \souw(A^{k_1}_{\sigma},\sigma)$, $\rho(i) = j$ if and only if $\sigma(i) = j$.  In particular, all points of the form $\{(i,\sigma(i)): i \in A^l_\sigma\}_{l\geq k_1}$ are contained in the frozen region.  It is possible the frozen region may contain some points outside of this region's complement as well, as shown in Figure \ref{figAB}. Note that by Lemma \ref{goodprops}, the frozen region also contains the region \[
W^+_{n,2\epsilon_1} = \{(x,y) \in R((0,0),(n,n)): x+y > (1+2\epsilon_1)n\},\] which contains no points of the form $(i,\sigma(i))$ and therefore no points of the form $(i,\rho(i)).$  Therefore $\mu_{\rho}(W^+_{2 \epsilon_1}) = 0$ and by Lemma \ref{oneside} $\mu_\rho(W_{\epsilon_2} < \epsilon_2)$.

\begin{figure}[htbp]
\centering
    \includegraphics[scale=.45]{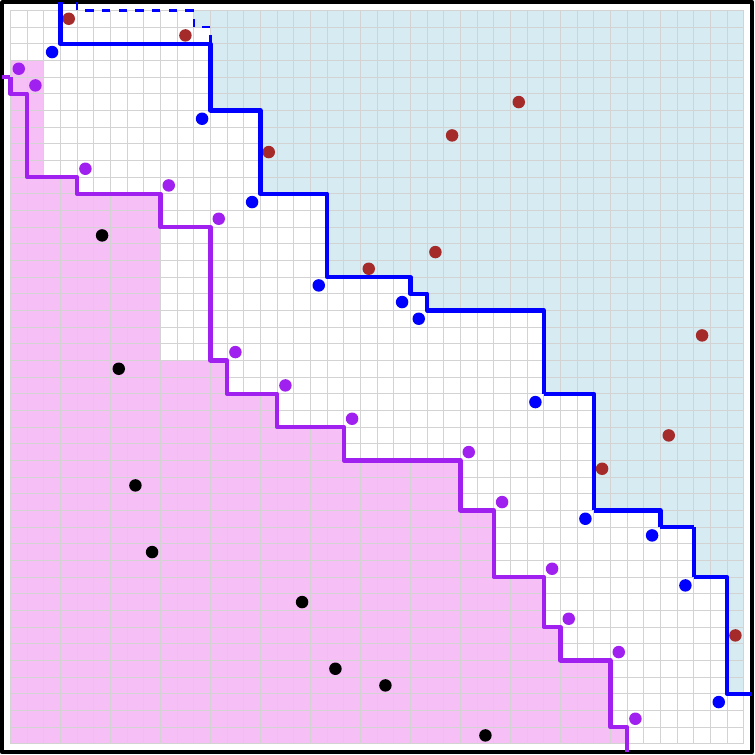} \qquad
    \includegraphics[scale=.45]{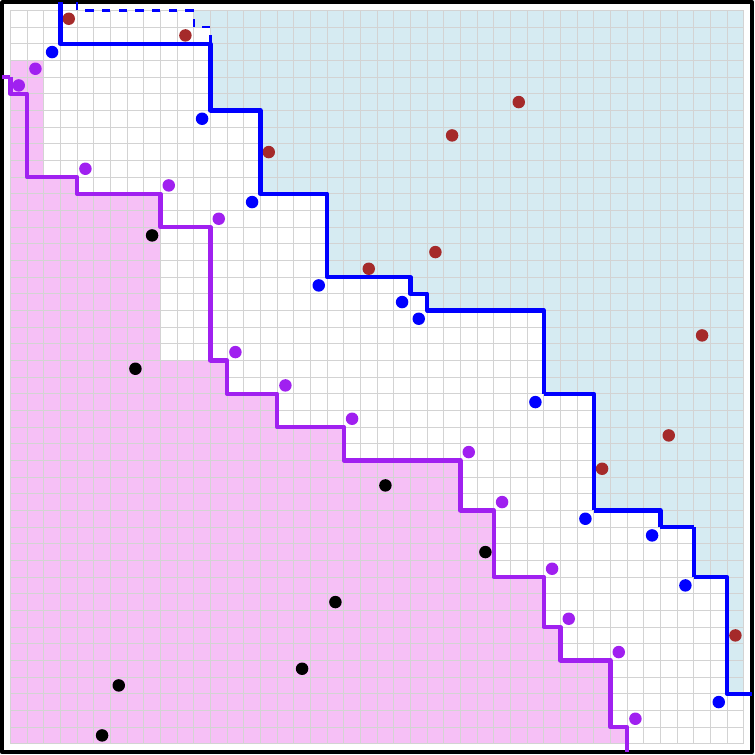}
    \caption{The left picture is the reverse complement of $\rho$ from Fig. \ref{figAB}, $\rho^{rc}\in \avn{I_{3}\oplus J_{2}}$.  The purple curve delineates the boundary of $\souw(B^{2}_{\rho^{rc}},\rho^{rc})$.  The frozen region is contained in the complement of the purple shaded region and does contain some points in the complement of $\souw(B^2_{\rho^{rc}},\rho^{rc})$.  The right picture is the image of $\rho^{rc}$ in $\avn{J_2 \oplus I_1 \oplus J_2}$ under the BWX bijection.  }
    \label{figCD}
\end{figure}

Now consider the decreasing set of points $\{(i,\sigma(i)): i \in A^{k_1+k_2}_\sigma = \{(i,\rho(i)): i \in A^{k_1 + k_2}_\sigma\}$.  If $i' \in A^{k_1+ k_2}_\sigma$ is in the interval $(2\epsilon_1 n,(1-2\epsilon_1) n)$, then from Lemma \ref{sequence}, there exists an increasing sequence of indices $i' = i_{k_1+k_2} < i_{k_1+k_2+1} < \dots < i_{k_1+k_2+k_3-1}$ with $i_l \in A^l_\sigma$ such that $(i_l,\sigma(i_l)) = (i_l,\rho(i_l))$ is increasing for $k_1+k_2\leq l < k_1 + k_2+k_3.$  Let $A = A^{k_1+k_2} \cap (2\epsilon_1 n, (1-2\epsilon_1 )n).$  Since $\rho$ avoids $J_{k_1} \oplus I_{k_2} \oplus I_{k_3}$, then $\rho$ will avoid $J_{k_1} \oplus I_{k_2}$ in the region $\souw(A,\rho)$.  

Consider the permutation $\rho^{rc}$ and the following region:
\[
F = \{(i,j) : (n+1-i,n+1-j) \in \souw(A,\rho)\}.
\]
Since $\rho$ avoids $J_{k_1} \oplus I_{k_2}$ in $\souw(A,\rho)$, the $\rho^{rc}$ will avoid $I_{k_2} \oplus J_{k_1}$ in $F$.  Thus for a box $(i,j) \in F \cup \partial F$, there are no occurrences of $I_{k_2} \oplus J_{k_1}$ strictly northeast of $(i,j)$ and the frozen region under the BWX bijection from $\rho^{rc}$ to $\pi^{rc}$ will contain all of $F \cup \partial F.$  We note that $F \cup \partial F$ will contain the region $W^+_{n,\epsilon_2}$ and $\mu_{\pi^{rc}}(W^+_{\epsilon_2}) = \mu_{\rho^{rc}}(W^+_{\epsilon_2}) < \mu_{\rho}(W_{\epsilon_2}) < \epsilon_2.$
Applying Lemma \ref{oneside} again we have that $\mu_{\pi}(W_\epsilon) = \mu_{\pi^{rc}}(W_{\epsilon}) < \epsilon,$ finishing the proof.
\end{proof}

\begin{proof}[Proof of Theorem \ref{mainthm}]
Fix $\epsilon >0$.  Choose $\epsilon_1$ small enough such that Lemma \ref{epi} applies.  If $\sigma_n \in \avne{I_{k_1}\oplus I_{k_2} \oplus I_{k_3}}{\epsilon_1}$, then $\pi_n\in \avn{J_{k_1} \oplus I_{k_2} \oplus J_{k_3}}$ satisfies $\mu_{\pi_n}(W_\epsilon ) < \epsilon.$  Let $\pi_n$ be chosen uniformly at random from the image of $\avne{I_{k_1} \oplus I_{k_2} \oplus I_{k_3}}{\epsilon_1} \subset \avn{J_{k_1}\oplus I_{k_2} \oplus J_{k_3}}$.  Then by Lemma \ref{epi}, $\mu_{\pi_n}(W_\epsilon)< \epsilon$, and thus by Lemma \ref{weakconv}, $\mu_{\pi_n} \to \muJ$ weakly.  By Lemma \ref{goodprops}, $\pi_n$ is the in the image of $\avne{I_{k_1} \oplus I_{k_2} \oplus I_{k_3}}{\epsilon_1}$ with probability tending to 1 as $n$ increases, thus we may conclude that $\mu_{\pi_n}$ converges in distribution to $\muJ$ with respect to the appropriate topology.    
\end{proof}

\section{Open Questions}

Our results only look at some classes of permutation that are in bijection with those avoiding a monotone increasing sequence.  However, Lemma \ref{oneside} could be useful for showing permuton limits of $\muJ$ for a larger set of pattern-avoiding classes.  Initially we attempted to explore $\avn{1342}$ using results developed by B\'ona \cite{bona1342}.  We could not show a permuton limit, but suspect convergence to $\muJ$ could be shown by understanding how the bijection in \cite{bona1342} with certain labeled trees could give a one sided bound on the permuton limit.  
It would be nice to understand more about the permuton limits for classes avoiding a single pattern of length 4.  Understanding permuton limits for $\avn{1342}$, $\avn{1324}$ and $\avn{2413}$ would complete the picture for permuton limits of classes avoiding a single pattern of size $4$.  We conjecture that both $\avn{1342}$ and $\avn{1324}$ have permuton limits of $\muJ$, whereas for $\avn{2413}$ the permuton limit is a bit of a mystery still.  Observing that permutations in $\avn{2413}$ have rotational symmetry shows that neither $\muJ$ nor $\muI$ could be the permuton limit for this case.  The permutation class $\avn{2413,3142}$ was shown to have a random permuton limit \cite{separable} which suggests the permuton limits for the larger classes ($\avn{2413}$ or $\avn{3142}$) might have more complicated structure than what we explore in this paper.      \\

\section*{Acknowledgements}
	This material is based upon work supported by the National Science Foundation Graduate Research Fellowship Program under Grant No. DGE-2040435.  The authors would like the thank a very generous referee who suggested many corrections and improvements.  We also would like to thank David Bevan for pointing out the conjectures in \cite{diagconj} to which our paper makes some partial progress.

\end{document}